\documentclass[reqno,11pt, a4paper,toc=flat]{amsart} 
\usepackage{setspace}
\usepackage[T1]{fontenc}
\usepackage{txfonts}

\usepackage[margin=3.0cm]{geometry}
\usepackage[utf8]{inputenc}
\usepackage[title]{appendix}
\usepackage{tikz}
\usepackage{tikz-cd}
\usepackage[default]{lato}
\usepackage[T1]{fontenc}   	
\usepackage{geometry}  
\usepackage{bold-extra}
\usefont{T1}{ptm}{b}{it}
\usepackage{graphicx}

\usepackage{graphicx}
\usepackage{amssymb}	
\usepackage{epsfig}	

\usetikzlibrary{positioning}
\usepackage[colorlinks=true,citecolor=blue,urlcolor=blue,linkcolor=blue]{hyperref}
\usepackage{graphicx}
\usepackage{url}

\usepackage{setspace}
\usepackage[T1]{fontenc}
\usepackage{txfonts}
\usepackage[margin=3.0cm]{geometry}
\usepackage{epsfig}
\def\scfig #1 #2 {\resizebox{#2}{!}{\includegraphics{#1}}}
\usepackage{graphics}
\usepackage{color}

\usepackage{color,soul}
\usepackage{color}
				
\usepackage{tikz-cd}					
\usepackage{tikz}
\usetikzlibrary{matrix}

\usepackage{amsmath,amsthm,amscd,amsfonts,amssymb,euscript,amsthm,amsfonts,amssymb,amscd}
\usepackage{graphics}
\usepackage{tikz-cd}
\usepackage{pdfpages}
\usepackage{mathtools}

\usepackage[widespace]{fourier}

\usepackage{color}
\usepackage{hyperref} \hypersetup{colorlinks=true,linkcolor=blue, citecolor=blue} 
\usepackage{fancyhdr}
\usepackage{fullpage}
\usepackage{curves}
\usepackage{pdfpages}
\usepackage{mathtools} 

\usepackage{frcursive}

\usepackage[all]{xypic}
\usepackage{latexsym}
\usepackage{color}
\theoremstyle{plain}

\numberwithin{equation}{section}

\newcommand{\spec}{{\rm Spec}\,}

\newtheorem{thm}[subsection]{Theorem}
\newtheorem{prop}[subsection]{Proposition}

\newtheorem{lema}[subsection]{Lemma}
\newtheorem{coro}[subsection]{Corollary}
\theoremstyle{definition}
\newtheorem{rema}[subsection]{Remark}
\newtheorem{defe}[subsection]{Definition}

\begin{document}

\input{amssym.def}

\title{A Torelli Type theorem for Nodal curves} \author{Suratno Basu and Sourav Das}
\address{SRM University, AP}
\email{suratno.b@srmap.edu.in}
\address{TIFR, Mumbai}
\email{sdas@math.tifr.res.in}
\date{16.6.2021}
\maketitle

\begin{abstract}
The moduli space of Gieseker vector bundles is a compactification of moduli of vector bundles on a nodal curve.
This moduli space has only normal crossing singularity and it provides a flat degeneration.
We prove a Torelli type theorem for a nodal curve using the moduli space of stable Gieseker vector bundles of fixed rank (strictly greater than $1$) and
fixed degree such that rank and degree are co-prime.
\end{abstract}

\section{Introduction}
Let $X$ be a compact connected Riemann surface of genus $g_{_X}$. The classical Torelli theorem says that
the isomorphism class of $X$ is uniquely determined by the isomorphism class of the Jacobian $Pic^0(X)$,
together with its canonical principal polarisation $\Theta$.

Suppose $g_{_X}\geq 2$. Given a line bundle $L$ on $X$,
let $M_{_{n,L}}(X)$ denote the moduli space of semi-stable vector bundles $E$ over $X$ of rank $n$ with
fixed odd determinant $\wedge^nE = L$. The Torelli problem for the moduli space $M_{_{n,L}}(X)$- namely the question whether the
isomorphism class of $M_{_{n,L}}(X)$ determines the isomorphism class of $X$- was studied by several authors.
If $g_{_X}\geq 2$ and $\tt{g.c.d}(rank(E),deg(\wedge^nE))=1$ then the
the Torelli problem for $M_{_{n,L}}(X)$ was solved by Mumford-Newstead \cite{DP} and Tyurin \cite{A} (in the case of rank $2$)
and Narasimahn-Ramanan \cite{MS} (for any rank).
Later assuming the genus $g_{_X}\geq 3$ Kouduvakis-Pantev \cite[Theorem E]{AT} completely solved the Torelli problem from $M_{_{n,L}}(X)$
without assuming the condition $\tt{g.c.d}(rank(E),deg(\wedge^nE))=1$.
We can also formulate Torelli problems for the various moduli spaces of semi stable parabolic bundles with fixed determinant in similar manner. 
In \cite{VIS}, Balaji et al. show that the isomorphism class of certain moduli space of
rank $2$ parabolic bundles with fixed determinant determines the isomorphism class of $X$ with marked points.
Recently, in \cite{DT II}, Alfaya and Gomez have generalised
this result for all rank and degree and for genus strictly greater than $3$.

In this article, we turn our attention to a Torelli type problem for a certain very interesting moduli space associated to a nodal curve, namely
the moduli space of stable Gieseker bundles. There are Torelli type theorems known for some moduli spaces associated to a nodal curve. In particular
in \cite{Y}, Namikawa and in \cite{J}, Carlson proved Torelli type theorem for curves with nodal singularity using the generalised Jacobian with an assumption
that the normalization of the nodal curve is not hyper-elliptic. In \cite{UI}, Bhosle and Biswas have proved a Torelli type theorem for a rational nodal curve
(i.e., arithmetic genus$=1$)
using the moduli of stable torsion-free sheaves. In \cite{S}, a Torelli type theorem was proved for certain moduli space of rank $2$ stable torsion-free sheaves
with fixed odd determinant over a reducible curve with two smooth components meeting at a node.
The moduli of semi-stable vector bundles over an irreducible nodal curve is not compact.
The moduli space of semi-stable Gieseker bundles provides a natural compactification of the moduli of semi-stable vector bundles on the nodal curve and
has been extensively studied
by several authors. The closed points of this moduli space correspond to (equivalence classes of) certain vector bundles (known as Gieseker bundles)
over certain semi-stable models of the nodal curve.
We ask the question whether this moduli space remembers the structure of the nodal curve, namely whether the isomorphism class this
moduli space determine the isomorphism class of the nodal curve. There is another approach to compactify the moduli of semi-stable vector bundles over an irreducible nodal curve
using semi-stable torsion-free sheaves.
However, there are some difficulties in proving Torelli type theorem using the moduli of semi-stable torsion-free sheaves over higher genus irreducible nodal curves.
As such, to the best of our knowledge, no Torelli type theorem is known for the moduli space of stable torsion-free sheaves of fixed rank and degree over
irreducible nodal curves.
In this paper, we obtain a Torelli type theorem for the moduli space of stable Gieseker vector bundles.
The main result of this paper is the following:

\begin{thm}
Let $Y$ and $Y'$ are two irreducible projective nodal curves (with single node) of genus $g\geq 2$ with single nodes $p$ and $p'$.
Let $\mathcal{M}_Y$ and $\mathcal{M}_{Y'}$ be the moduli space of stable Gieseker vector bundles of rank $2$ and
degree $d$ and $(2,d)=1$. If $\mathcal{M}_Y\cong \mathcal{M}_{Y'}$, then $Y\cong Y'$.

Suppose $Y$ and $Y'$ are curves of genus strictly greater than $3$. Let $\mathcal{M}_{_{Y,n,d}}$ and $\mathcal{M}_{_{Y',n,d}}$ be the moduli space of stable Gieseker vector bundles of rank $n$ and
degree $d$ and $(n,d)=1$. If $\mathcal{M}_{_{Y,n,d}}\cong \mathcal{M}_{_{Y',n,d}}$, then $Y\cong Y'$.
\end{thm}

The strategy of the proof is the following :\\
The moduli spaces $\mathcal{M}_Y$ and $\mathcal{M}_{Y'}$ have normal crossing singularity.
Any isomorphism between $\mathcal{M}_Y$ and $\mathcal{M}_{Y'}$ induces an isomorphism between their singular loci and
this isomorphism induces an isomorphism between the singular loci of the singular loci and so on.
Proceeding in this manner we get an isomorphism between the most singular loci of $\mathcal{M}_Y$ and $\mathcal{M}_{Y'}$.
Let us denote by $\pi:X\rightarrow Y$ and $\pi':X'\rightarrow Y'$
the normalizations of $Y$ and $Y'$ and let $\{p_1, p_2\}$ and $\{p'_1, p'_2\}$ be the pre-images of $p$ and $p'$
under the normalization maps.
Then we identify the most singular loci of $\mathcal{M}_Y$ and $\mathcal{M}_{Y'}$
with the moduli spaces of stable parabolic vector bundles on the normalizations $X$
and $X'$ equipped with full flags at $\{p_1, p_2\}$ and $\{p'_1, p'_2\}$.
Therefore we get an isomorphism between these moduli spaces of parabolic vector bundles.
This isomorphism further induces an isomorphism between the moduli spaces of parabolic vector bundles with fixed determinants.
Then using the Torelli theorem for moduli spaces of parabolic vector bundles, as obtained by \cite{VIS} and \cite{DT II}, we achieve an isomorphism $\phi:X\rightarrow X'$ which maps the divisor $\{p_1, p_2\}$ to $\{p'_1, p'_2\}$.
This induces an isomorphism between $Y$ and $Y'$.
In a recent preprint of \cite{SAI}, Basu et al. a Torelli type theorem for certain
Moduli spaces of rank $2$ Gieseker bundles with fixed odd determinant over a nodal curve is obtained with an
extra assumption that the normalization of the nodal curve is non-hyperelliptic. Their technique is different from this paper.

\subsubsection*{Acknowledgement} We thank Dr. Krishanu Dan for helpful discussions. We thank the referee for his/her valuable comments/suggestions
which helped us to improve the exposition of this article.

\section*{Notations and Conventions}

We fix some notations and conventions which will be used freely in this paper.

\begin{itemize}

\item All schemes will be defined over the complex numbers $\mathbb{C}$.

\item We will denote by $Y$ a nodal curve with a single node $p$ and by $q: X\rightarrow Y$ the normalization of $Y$.
We will denote by $Y_k$ a Gieseker curve of length $k$ (\ref{gis}) for $k\geq 1$ and $Y_0:=Y$.

\item We fix a pair $(n,d)$, of integers where $n > 0$ and $d$ arbitrary, where $n$ will be the rank and $d$ the degree of the bundles.
Throughout the paper we assume that $\text{g.c.d}(n,d) = 1$.
\item We say that a variety $X$ has normal crossing singularity at a point $x\in X$ if the analytic local ring $\widehat{\mathcal O}_{_{X,x}}\cong \frac{k[|x_1,\dots,x_{r+1},x_{r+2},\dots, x_n|]}{x_1\cdots x_{r+1}}$ for some positive integer $r$. Furthermore, if every point of $X$ satisfies the condition then we call $X$ to be a variety with normal crossing singularity.
\item Let $X$ be a smooth quasiprojective variety with a divisor $D$. We say that $D$ is a normal crossing divisor of $X$ if for any point $d\in D$, the divisor $D$ is defined by the equation $x_1\cdots x_r$ in the analytic local ring $\widehat{\mathcal O}_{_{X,d}}\cong \mathbb C[|x_1,\dots, x_r, x_{r+1,\dots, x_n}|]$, with respect to some choice of analytic local coordinates $\{x_1,\dots, x_r,x_{r+1},\dots, x_n\}$ of $X$ at the point $d$.

\end{itemize}

\section{Basic Definitions and results}

Let $Y$ be a nodal curve of genus $g$ with a single node $p$.
Let $q:X\rightarrow Y$ be the normalization map. Note $X$ is a smooth curve of genus $g-1$.
We denote the points in the inverse image of $p$ under the map $q$ by $p_{_1}$ and $p_{_2}$ and by $D$
the divisor $\{p_{_1},p_{_2}\}$.

\begin{defe}\label{gis}{Gieseker Curves:}
\begin{enumerate}
\item We call a scheme $R$ a chain of projective lines if
$R =\cup_{_{i=1}}^{^k} R_{_i}$, $R_{_i}\cong \mathbb{P}^{^1}, R_{_i}\cap R_{_j}$ (for distinct $i,j$) is a single point
if $|i-j|=1$ and otherwise empty. We call $k$ the length of $R$.

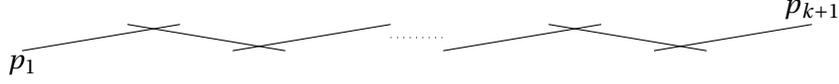
\vspace{1cm}
\begin{tikzpicture}[overlay, xshift= 1cm,scale=0.70]
\draw node[yshift=-1ex]{$p_1$}(0,0) -- (3,.5) node[yshift=1ex]{}; \draw (2,.5) -- (5,0)node[yshift=-1ex]{}; \draw (4,0) -- (7,.5); \draw[dotted] (7,.25) -- (8,.25);\draw (8,0) -- node[yshift=1.5ex]{}(11,.5); \draw (10,.5) -- (13,0)node[yshift=-1ex]{}; \draw (12,0) -- (15,.5) node[yshift=1ex]{$p_{k+1}$};
\end{tikzpicture}
\vspace{1cm}

\item A Gieseker curve $Y_{_k}$ is a union of $X$ with the chain $R$ of $\mathbb{P}^1$s of length $k$ such that $X\cap R=\{p_{_1}, p_{_{k+1}}\}$ and $p_{_1},p_{_{k+1}}$ are two smooth points in the extremal irreducible components of $R$. Set $Y_0:=Y$.
\end{enumerate}
\end{defe}

\begin{rema}\label{Rem111}
Note that there is a morphism $\pi_k: Y_k\rightarrow Y$ which is isomorphism outside $R$ and the image of $R$ is the node $p$.
\end{rema}

\begin{defe}{Family of Gieseker curves: }\label{Modi} Let $T$ be any scheme over $\mathbb C$.
Then a family of Gieseker curves over $T$ is a family of curves $\mathcal X_T$ with a morphism $\mathcal X_T\rightarrow Y\times T$
such that for all $t\in T$ the morphism $\mathcal X_{_{T,t}}\rightarrow Y\times t$ is same as $\pi_{_k}$
for some integer $k$. We also say that $ \mathcal X_T$ is a modification of $Y\times T$.
\end{defe}

\begin{defe}{Gieseker bundle:} \label{DEF}
\begin{enumerate}
\item Let $\mathcal E$ be a vector bundle of rank $n$ on $R$. One knows that
$\mathcal E|_{R_i} =\oplus_{j=1}^n \mathcal{O}(a_{ij} ), a_{ij} \in \mathbb{Z}$. We say that $\mathcal E$ is positive
if $a_{_{ij}}\geq 0$ for all $i,j$. We say that $\mathcal E$ is strictly positive if it is positive and for every $i$,
there is a $j$ such that $a_{ij}>0$. We say $\mathcal E$ is standard
if it is positive and $a_{ij}\leq 1$ for all $i, j$ and strictly standard if, moreover, it is strictly positive.
\item A Gieseker vector bundle $\mathcal E$ of rank $n$ is a vector bundle on $Y_k$ such that $\mathcal E|_{_R}$ is a strictly positive bundle and
${\pi_k}_*\mathcal E$ is a torsion-free sheaf on the nodal curve $Y$.
\item We call a vector bundle on a chain $R$ of $\mathbb{P}^1$'s a Gieseker vector bundle
if it is the restriction of a Gieseker vector bundle on some Gieseker curve whose chain of $\mathbb P^1$'s is $R$.

\end{enumerate}
\end{defe}

\begin{rema}\label{type}
For any Gieseker vector bundle $\mathcal E$ of rank $n$ over
$Y_k$, with $0\leq k\leq n$, we have $\pi_{k~~*}\mathcal{E}$ is a torsion free sheaf. This property is equivalent to the fact
that $H^0(R, \mathcal{E}|_{_R}(-p_{_1}-p_{_2}))=0$ (see \cite[Proposition 5]{DC}). Using this property
we can derive that if $\mathcal E$ is a Gieseker vector bundle of rank $n$ over $Y_n$ then
$\mathcal E|_{R_i}\simeq \mathcal{O}^{n-1}\oplus \mathcal{O}(1)$, $i=1,\cdots, n$ (see \cite[Proposition 5, (ii) and B]{DC}). Also notice that if $\mathcal E$ is a Gieseker vector bundle of rank $n$ on a Gieseker curve $Y_k$ then $0\leq k\leq n$.
\end{rema}

Let $T$ be any scheme. By a family of Gieseker vector bundles parametrised by $T$, we mean a vector bundle $\mathcal E_{T}$ over $\mathcal X_T$ such that $\mathcal E_T|_{\mathcal {X}_{T,t}}$ is a Gieseker vector bundle for each $t\in T$.

\begin{defe}\label{EqGie}
Two families of Gieseker vector bundles $\mathcal E_T$ and $\mathcal E'_T$ over two families of Gieseker curves $\mathcal{X}_T$ and $\mathcal{X}'_T$
respectively are said to be equivalent if there exists a $T$-isomorphism $\phi_T$ with
\begin{center}
\begin{tikzcd}
\mathcal{X}_T\arrow{rr}{\phi_T}\arrow{dr} && \mathcal{X}'_T \arrow{dl}\\
& Y\times T
\end{tikzcd}
\end{center}
commutative and a line bundle $L$ on $\mathcal{X}'_T$, pulled back from $T$ such that $\mathcal E_T\cong \phi_T^*(\mathcal E'_T\otimes L)$.
\end{defe}

\begin{defe}\label{stabilityDEF}
A Gieseker vector bundle $(Y_k, \mathcal E)$ is called stable if the torsion-free sheaf $(\pi_k)_* \mathcal E$ is stable.
\end{defe}

\begin{defe}\label{G}We define a functor $G$ of stable Gieseker vector bundles
(we drop $(n,d)$ from the notation):
$$
G: \tt{Sch}/\mathbb{C}\rightarrow \tt{Sets}
$$
\begin{equation}\label{Func201}
T\mapsto \left\{
\begin{array}{@{}ll@{}}
\text{equivalence classes of stable} \\
\text{Gieseker vector bundles of}\\
\text{ rank $r$ and degree $d$ over} T
\end{array}\right\}
\end{equation}
\end{defe}

\subsection{The moduli space of stable Gieseker vector bundles:}\label{Step 2}
There exists an irreducible, reduced projective scheme $G(n,d)$ which represents the functor $G$ (\ref{Func201}). Furthermore, $G(n,d)$ has normal crossing
singularities. We refer to $G(n,d)$ as the moduli space of stable Gieseker bundles of rank $n$ and degree $d$. The closed points of
$G(n,d)$ correspond to the equivalence classes of pairs $(Y_k, \mathcal E)$, where $Y_k$ is a Gieseker curve, $0\leq k\leq \text{rank}=n$ and $\mathcal E$ is a stable Gieseker
vector bundle on $Y_k$ of rank $n$ and degree $d$.

We shall very briefly summarize the main steps of the construction of $G(n,d)$. We refer to \cite[Proposition 6, Proposition 7]{DC} for the construction of the moduli space. We begin with recalling the definition of the Gieseker functor $\mathcal G$ for convenience.

Let us choose and fix an ample line bundle $\mathcal O_Y(1)$ on the nodal curve $Y$. Let us choose a large integer $m$, such that for any stable torsion-free sheaf $\mathcal F$ of rank $n$ and degree $d$ the sheaf $\mathcal F\otimes \mathcal O_Y(m)$ is globally generated and $H^1(Y, \mathcal F\otimes \mathcal O_Y(m))=0$ (It is possible because all the stable torsion-free sheaves of fixed rank and degree on $Y$ lies in a bounded family).

\begin{defe}\label{GieFunc}
Let $\mathcal{G} = \mathcal{G}(n, d)$ be the functor (called the Gieseker functor) defined as follows:

\begin{center}
$\mathcal G: \tt{Sch/\mathbb C}\rightarrow \tt{Sets}$
\end{center}

$\mathcal G(T)= ~~\text{set of closed subschemes} ~~\Delta\hookrightarrow Y\times T \times Gr(m,n)$ such that

\begin{enumerate}
\item the induced projection map $ \Delta\rightarrow T \times Gr(m, n)$ is a closed immersion. We denote by $\mathcal E$ the rank $n$ vector bundle on $\Delta$ which is the pull-back of the tautological rank $n$ quotient bundle on $Gr(m, n)$,
\item the projection $\Delta \rightarrow T$ is a flat family of curves $\Delta_t (t \in T)$ such that $\Delta_t$ is a curve of the form $Y_k$. Besides, the canonical map $\Delta_t \rightarrow Y$ is the map $\pi_k: Y_k(= \Delta_t) \rightarrow Y$ that we have been considering,
\item the vector bundle $\mathcal E_t$ on $\Delta_t$ is of degree $d$ (and rank $n$) with $d = m+ n(g- 1)$.

\item By the definition of $\mathcal E$, we get a quotient representation $\mathcal{O}^m_{_{\Delta_t}}\rightarrow \mathcal E_t$,
and we assume that this induces an isomorphism
$H^0(\mathcal{O}^m_{_{\Delta_t}})\xrightarrow{\cong} H^0(\mathcal E_t)$
In particular, dim $H^0(\mathcal E_t)=m$. It follows that
$H^1(\mathcal E_t) = 0$.
\end{enumerate}
\end{defe}

The main steps of the construction can be summarized as follows:

\begin{enumerate}
\item[Step 1:] It is shown in (\cite[Definition 5, Proposition 6]{DC}) that the Gieseker functor $\mathcal{G}$ is represented by
a $PGL(m)$-invariant open subscheme $\mathcal Y$ of the Hilbert scheme of curves $Hilb^P(Y\times Gr(m,n))$, where $P$ is the Hilbert polynomial of a Gieseker vector bundle of rank $n$ and degree $d$. Furthermore, $\mathcal{Y}$
is irreducible and has normal crossing singularities.

\item[Step 2:] Let $\mathcal R$ be a suitable Quot scheme for torsion-free sheaves of rank $n$ and degree $d$ on the nodal curve $Y$ (\cite[Page 179]{DC}), which contains all the stable torsion-free sheaves. There is a proper birational morphism $\theta: \mathcal Y\rightarrow \mathcal R$ (\cite[Proposition 10]{DC}).
Let $\mathcal R^{st}$ denote the open subscheme containing the stable torsion-free sheaves.
Consider $\mathcal Y^{st}:=\theta^{-1}(\mathcal R^{st})$. Then it can be seen that $\mathcal Y^{st}$ is a
$PGL(m)$-invariant subscheme and $PGL(m)$ acts freely on $\mathcal Y^{st}$.
Then it is shown in (\cite[Page 179,180]{DC}) that the quotient $\mathcal Y^{st}\parallelslant PGL(m)$ exists as a projective scheme and represents
the functor $G$ (\ref{G}). Since $PGL(m)$ acts freely on $\mathcal Y^{st}$ the quotient also acquires the normal crossing singularity.
\end{enumerate}

We recall the following definition from \cite[Appendix: Local theory, III]{DC}:

\begin{defe} We define a functor
$$
\mathcal{G}': \tt{Sch/\mathbb C}\rightarrow \tt{Sets}
$$
\begin{equation}
T\mapsto \left\{
\begin{array}{@{}ll@{}}
\text{ Isomorphism classes of }~~\text{Gieseker curves $\Delta\rightarrow T$ over}~~T
\end{array}\right\}
\end{equation}
\end{defe}

We shall now briefly outline the main steps of the proof of the fact that $\mathcal Y$ has normal crossing singularities.
\begin{enumerate}

\item{Step 1:} Let $\widehat{\mathcal G}'_k$ denote the functor of Artin local rings whose $T$ valued points, for any Artin local scheme $T$, is the set of isomorphism classes of Gieseker curves over $T$, whose closed fiber is isomorphic to the Gieseker curve $Y_k$. In \cite[Appendix: Local theory, IV]{DC}, it is shown that $\widehat{\mathcal{G}'_k}\cong \frac{k[[t_1,....,t_k]]}{t_1\dots t_k}$ i.e., there exists a miniversal effective family of Gieseker curves over $\frac{k[[t_1,....,t_k]]}{t_1\dots t_k}$. Here $t_i=0$ is the equation of the $i$-th node, i.e, in a local deformation where $t_i=0$ the $i$-th node is not deformed, it remains as a node. On the other hand, $t_i\neq 0$ implies otherwise.

\item{Step 2:} The natural transformation $\mathcal{G}\rightarrow \mathcal G'$ is formally smooth.
The Gieseker-functor $\mathcal G$ is represented by the variety $\mathcal Y$ (\cite[Proposition 6]{DC}).
Therefore, from step 1 it follows that $\mathcal Y$ is a variety with only normal crossing singularity.

\end{enumerate}

\subsection{The relative construction}\label{RelCons}

Let us choose a family of curves
\begin{equation}\label{Curve}
\mathcal X\rightarrow S
\end{equation}
over a discrete valuation ring $S$ such that
\begin{enumerate}
\item the generic fiber is isomorphic to a smooth projective curve
\item the special fiber is isomorphic to the nodal curve $Y$
\item the total space $\mathcal X$ is regular.
\end{enumerate}

The existence of such a family follows from \cite[Theorem B.2 and Corollary B.3, Appendix B]{M}.

We recall the following facts:
\begin{enumerate}
\item There exists a family of varieties $\mathcal Y^{\tt{st}}_S$ over $S$ such that the fiber over the closed point is isomorphic to $\mathcal Y^{\tt{st}}$ and it is a normal crossing divisor in $\mathcal Y^{\tt{st}}_S$ (\cite[Proposition 8]{DC}).
\item There exists a family of varieties $G(n,d)_S$ over $S$ such that the fiber over the closed point is isomorphic to $G(n,d)$ and it is a normal crossing divisor in $G(n,d)_S$. In fact, $G(n,d)_S$ is a GIT quotient (relative to $S$) of $\mathcal Y^{\tt{st}}_S$ by a free action of the group $PGL(m)$. We refer to \cite[Theorem 2]{DC} for further details.
\end{enumerate}

\section{A stratification of the moduli space of stable Gieseker vector bundles}

\begin{lema}\label{NormalCrossings}
Let $X$ be a smooth variety with a normal crossing divisor $X^0$. Consider the following stratification of $X^0$ by the successive singular loci i.e., the natural stratification
\begin{equation}
X^0\supset X^1\supset\dots\supset X^i\supset \dots \supset X^n\supset X^{n+1}=\emptyset
\end{equation}
where $X^{i+1}:=$ the singular locus of $X^i$, for each $0\leq i\leq n$.
Let us denote by $\pi: \tilde X^0\rightarrow X^0$ the normalization morphism. Then for every $0\leq i\leq n$,
\begin{equation}
X^i=\{x\in X^0~~| ~~\text{cardinality of the set}~~\pi^{-1}(x)\geq i+1\}
\end{equation}
\begin{equation}
\text{and}~~X^i\setminus X^{i+1}=\{x\in X~~| ~~\text{cardinality of the set}~~\pi^{-1}(x)= i+1\}
\end{equation}

Moreover, the variety $X^{i+1}$ is a Zariski-closed subvariety of $X^i$ of pure codimension $1$, if non-empty. The variety $X^n$ is a smooth variety of codimension $n$ in $X^0$. We call $X^n$ the most singular locus of $X^0$.
\end{lema}
\begin{proof}
The proof follows from \cite[Lemma 2.1]{Br}.
\end{proof}

Let us fix a family of curves $\mathcal X$ over a discrete valuation ring $S$ as in subsection \ref{RelCons}. As we have discussed in the same subsection, there exists a family of varieties $G(n,d)_S$ over $S$ such that the fiber over the closed point is isomorphic to $G(n,d)$ and it is a normal crossing divisor in $G(n,d)_S$. Therefore, using Lemma \ref{NormalCrossings}, we see that the natural stratification of the moduli of stable Gieseker vector bundles
\[
\mathcal M^0:=G(n,d)\supset \mathcal M^1\supset \dots\supset \mathcal M^n \supset \mathcal M^{n+1}:=\emptyset,
\]
where $\mathcal M^{r+1}:=$the singular locus of $\mathcal M^r$ is for every $0\leq r\leq n$ has the following property:
\[
\mathcal M^i=\{x\in \mathcal M^0~~| ~~\text{cardinality of the set }~~\pi^{-1}(x)\geq i+1\}~~\text{ for every }~~0\leq i\leq n.
\]

\begin{prop}\label{Strat}
The closed points of $\mathcal M^r$ correspond to the equivalence classes of stable Gieseker bundles $(Y_k, \mathcal E)$, where $n\geq k\geq r$. In particular, $\mathcal M^n$ is a smooth projective variety of dimension $n^2(g-1)+1-n$, whose closed points correspond to the equivalence classes of stable Gieseker vector bundles $(Y_n, \mathcal E)$ of rank $n$ and degree $d$.
\begin{proof}
Let us first recall that the moduli space $\mathcal M^0:=G(n,d)$ is a GIT quotient of $\mathcal Y^{\tt{st}}$ by the free action of $PGL(m)$ (subsection \ref{Step 2}). Moreover, $\mathcal Y^{\tt{st}}$ is a normal crossing divisor in the family of varieties $\mathcal Y^{\tt{st}}_S$ (subsection \ref{RelCons}). Therefore, from lemma \ref{NormalCrossings}, it follows that the stratification of $\mathcal{Y^{\tt{st}}}$ by its successive singular loci:
\begin{equation}
\mathcal{Y}^{0,\tt{st}}=\mathcal Y^{\tt{st}}\supset \dots \supset \mathcal{Y}^{n,\tt{st}}\supset \mathcal{Y}^{n+1,\tt{st}}:=\emptyset,
\end{equation}

has the property that $\mathcal Y^{i+1,\tt{st}}$ is a closed subvariety of $\mathcal Y^{i,\tt{st}}$ of pure codimension $1$, if non-empty and
\begin{equation}\label{Final}
\mathcal Y^{i,\tt{st}}=\{y\in \mathcal Y^{\tt{st}}| ~~\text{cardinality of }~~\pi^{-1}(y)\geq i+1\},
\end{equation}
where $\pi: \tilde{\mathcal Y^{\tt{st}}}\rightarrow \mathcal Y^{\tt{st}}$ is the normalization morphism.

Now since the variety $\mathcal Y$ represents the functor $\mathcal G$ there is a universal curve $\Delta \hookrightarrow Y\times \mathcal Y\times Gr(m,n)$ and
the singular locus of the projection $\Delta\rightarrow \mathcal Y$ is a scheme defined by the vanishing of the first fitting ideal of $\Omega^1_{_{\Delta/\mathcal Y}}$ and in fact it is the normalization of $\mathcal Y$(\cite[proof of Theorem 4.9]{I}).
Moreover, a $S$-valued point of $\tilde{\mathcal Y}$ is a $S$-valued point of $\mathcal Y$ i.e.,
$\Delta \hookrightarrow Y \times S\times Gr(m,n)$ plus a section $S \rightarrow \Delta$ of the projection $p : \Delta \rightarrow S$,
meeting $\Delta$ in the singular locus of the morphism $p$ (\cite[proof of Theorem 4.9]{I}).
It follows that a fibre of the morphism $\pi: \tilde{\mathcal Y}\rightarrow \mathcal Y$ over a given point $y:=(Y_k, \mathcal E)$ of $\mathcal Y$ is the set of triples $\{(Y_k, \mathcal E, \tilde{y})\}$, where $\tilde{y}$ is a marked node of $Y_k$. Notice that the set $\pi^{-1}(Y_k, \mathcal E)$ has cardinality $k+1$, since $Y_k$ has exactly $k+1$ nodes. Since $\mathcal Y^{\tt{st}}$ is an open subset of $\mathcal Y$ the normalization of $\mathcal Y^{\tt{st}}$ is just the preimage of it under the normalization $\tilde{\mathcal Y}\rightarrow \mathcal Y$. Therefore from equation \ref{Final}, it follows that a closed point of $\mathcal Y^{k,\tt{st}}$ corresponds to an element $(Y_i, \mathcal E)\in \mathcal Y^{\tt{st}}$, where $i\geq k$. Since the length $k$ of the Gieseker curve $Y_k$ must be bounded by the rank (remark \ref{type}) therefore it follows that $\mathcal Y^{k,\tt{st}}=\emptyset$ for all $k>n$. In particular, $\mathcal Y^{n,\tt{st}}$ is a smooth quasi-projective variety.

From this description it is also clear that the closed subvarieties $\mathcal Y^{k,\tt{st}}$ are also stable under the action of $PGL(m)$. Now since the good quotient $\mathcal{M}^0:=\mathcal{Y}^{0,\tt{st}}\parallelslant PGL(m)$ exists and $\mathcal Y^{0,\tt{st}}\rightarrow \mathcal M^0$ is a principal $PGL(m)$ bundle,
the good quotients $\mathcal{M}^k:=\mathcal{Y}^{k,\tt{st}}\parallelslant PGL(m)$ also exist. Moreover,
the singular locus of $\mathcal{M}^k$ is $\mathcal{M}^{k+1}$ and the closed points of $M^k$ correspond to the equivalence classes of stable Gieseker bundles $(Y_i, \mathcal E)$, where $i\geq k$.
In particular, $\mathcal M^n$ is a smooth projective variety whose closed points are equivalence classes of stable Gieseker bundles $(Y_n, \mathcal E)$.
\end{proof}

\end{prop}


\section{The parabolic vector bundle associated to a Gieseker vector bundle}
In this section, we will identify the "most singular locus" of $\mathcal{M}^0$ with the moduli space of certain vector bundles
with full-flag parabolic structures. We need the following basic lemma. 

\begin{lema}\label{Sec}
Let us fix an integer $n\geq 2$. Let $T$ be a smooth variety and $\mathcal{X}_T$ be a modification (\ref{Modi}) of $Y\times T$ over $T$
such that for all $t\in T$ the fibre $\mathcal{X}_{T,t}$ is isomorphic to the fixed curve $Y_n$. Then
\begin{enumerate}
\item Let $X\rightarrow Y$ be the normalization. Then $X\times T$ is the
the normalization of $Y\times T$ and it is an irreducible component of $	\mathcal{X}_T$
\item $\mathcal{X}_T$ has only normal crossing singularity and also every irreducible component of $\mathcal{X}_T$ is smooth and flat over $T$
\item there are $n+1$ disjoint sections $P_1,...,P_{n+1}$ of $\mathcal X_T\rightarrow T$ which are
the nodal loci of $ \mathcal{X}_T$ and the singular loci of
the morphism $ \mathcal{X}_T\rightarrow T$.
\item $\mathcal{X}_T$ has $n+1$ smooth irreducible components.
There is one component $\mathcal{X}^0$ which is the normalization of $Y\times T$
and therefore isomorphic to $X\times T$. There are $n$ other irreducible
components $ \mathcal{X}^1,\dots, \mathcal{X}	^n$ each of
which is a $ \mathbb{P}^1$-bundle over $T$.
Also $\mathcal{X}^0\cap\mathcal{X}^1$, $\mathcal{X}^0\cap\mathcal{X}^n$
and $\mathcal{X}^i\cap\mathcal{X}^{i+1}$ for $1\leq i\leq n-1$ are the nodal loci of $\mathcal{X}_T$ and
they are all isomorphic to $T$ under the projection map to $T$(follows from ($3$)).
\end{enumerate}
\end{lema}
\begin{proof}
Suppose $\eta=\spec~K$ be the generic point of $T$. Then $\mathcal{X}_{\eta}$ is isomorphic to the curve $Y_n$ defined over $K$.
Therefore using flatness of the map $\mathcal{X}_{T}\rightarrow T$, we see that $\mathcal{X}_{T}$ has $n+1$ irreducible components.
We denote them by $\mathcal{X}_0,\dots, \mathcal{X}_n$.

Let $t$ be a closed point in $T$ and let $p_i$ be a node in $\mathcal{X}_{t}$.
Then we see that we have a flat deformation of the node $p_i$ given by $\mathcal{X}_{T}$
over the analytic local ring $\widehat{\mathcal{O}}_{_{T,t}}$ of $T$ at $t$.

From Schessinger's theory it follows that given any flat deformation $Z$ of $Y_n$ over $Spec ~A$
(where $A$ is an Artinian $k$-algebra), there are $a_1,....,a_n\in A$ so that at the $i$-th node $N_i=R_i\cap R_{i+1}$ we have $\widehat{\mathcal{O}}_{_{Z,N_i}}\cong \frac{A[|x,y|]}{(xy-a_i)}$. The $a_i$ are determined upto a unit, so we refer to $a_i=0$ as the equation of the $i$-th node.

Using this it is easy to see that the analytic local ring $\widehat{\mathcal{O}}_{_{\mathcal{X}_{T},N_i}}=\frac{\widehat{\mathcal{O}}_{_{T,t}}[|x,y|]}{xy}$ and therefore it is formally smooth over $\frac{\mathbb{C}[|x,y|]}{xy}$, because $T$ is assumed to be smooth. Hence we conclude that the singularities of $\mathcal{X}_{T}$ are normal crossing. Therefore the irreducible components of $\mathcal{X}_{T}$ are flat over $T$ i.e., the morphisms $\mathcal{X}^i\rightarrow T$ are flat. Since these morphisms are also proper therefore they are faithfully flat.

Since $T$ is smooth and all fibres of $\mathcal{X}_T\rightarrow T$ have the same geometric genus we can apply
\cite[Theorem 1.3.2]{B} to normalise simultaneously the fibres of the map $\mathcal{X}_T\rightarrow T$.
We obtain a commutative diagram:

\begin{center}
\begin{tikzcd}
\tilde{\mathcal{X}_T}\arrow{rr}\arrow{dr} && \mathcal{X}_T \arrow{dl}\\
& T
\end{tikzcd}
\end{center}

where $\tilde{\mathcal{X}_T}\rightarrow T$ is a smooth projective family of curves and for each $k$-rational point $t \in T$
the induced morphism $\tilde{\mathcal{X}_{T,t}}\rightarrow \mathcal{X}_{T,t}$ is the normalization for all $t\in T$.
It follows that $\tilde{\mathcal{X}_T}$ is smooth and it is the normalization of $\mathcal{X}_T$.
Moreover $\tilde{\mathcal{X}_T}$ is the disjoint union of the irreducible components $\mathcal{X}^i$ of $\mathcal{X}_T$
and $\mathcal{X}^i\rightarrow T$ is smooth for all $i$. This proves (2).

Let $\mathcal{X}^0$ be the irreducible component whose fibre over every point is $X$. We have proved that $\mathcal{X}^0$ is smooth.
Therefore we conclude that $\mathcal{X}^0$ is the normalization of $Y\times T$ and hence it is isomorphic to $X\times T$.
This proves (1).

Now consider again the curve $\mathcal{X}_{\eta}$. It has $n+1$ distinct nodes $p_1,....p_{n+1}$. Using \cite[Lemma 50.20.1,
Section 30.10]{SP} and the fact that every fibre of $\mathcal{X}_T$ has $n+1$ distinct nodes we conclude $(3)$ and $(4)$.
\end{proof}

\subsection{Construction of the flag:}
Throughout this subsection, we assume $T$ is a smooth variety.
Let $\mathcal{R}_T$ be a family of smooth curves of genus $0$ over $T$ i.e.,
$\pi:\mathcal{R}_T\rightarrow T$ is flat and every fibre is isomorphic to $\mathbb{P} ^1$.
Also assume that we are given two disjoint sections $p_1$ and $p_2$.
Let $\mathcal{E}$ be a vector bundle on $\mathcal{R}_T$ such that restriction to every fibre is Gieseker vector bundle (Definition\ref{DEF}).
Given any sub-sheaf $V_{_{1}}$ of $\mathcal{E}_{_{p_1}}$ we will produce a natural sub-sheaf of $\mathcal{E}_{p_2}$.

Let us consider the following diagram:

\begin{equation}\label{4.1}
\begin{tikzcd}
& R^0\pi_*\mathcal{E}\arrow{dl}{q_1}\arrow{dr}{q_2}\\
\mathcal{E}_{p_1} && \mathcal{E}_{p_2}
\end{tikzcd}
\end{equation}
By checking at each geometric point $t\in T$, we see that the two arrows are surjective
(this is because the vector bundle $\mathcal{E}_{_t}$ is strictly standard $\forall t\in T$).

Consider the sub-sheaves $q_2(q_1^{-1}(V_{_1}))$ and $q_1(q_2^{-1}(V_{_2}))$ of $\mathcal{E}_{p_2}$ and $\mathcal{E}_{p_1}$
respectively. If $V_1$ is a saturated subsheaf of $\mathcal E_{p_1}$ then $q_2(q_1^{-1})(V_1)$
is also saturated subsheaf of $\mathcal E_{p_2}$. Similarly, if $V_2$ is a saturated subsheaf of $\mathcal E_{p_2}$
then $q_1(q_2^{-1})(V_2)$
is also saturated subsheaf of $\mathcal E_{p_1}$.

\begin{lema}\label{Int}
Let $\mathcal R$ be a chain of $\mathbb{P}^1$'s of length $n$ (definition \ref{gis}). Let $\mathcal E$ be a Gieseker vector bundle on $\mathcal R$.

Consider the following diagram:

\begin{equation}
\begin{tikzcd}
& H^0(R_{i+j}, \mathcal{E}|_{_{R_{i+j}}})\arrow{dl}{q_{i+j}}\arrow{dr}{q_{i+j+1}}\\
\mathcal{E}_{p_{i+j}} && \mathcal{E}_{p_{i+j+1}}
\end{tikzcd}
\end{equation}

Let us define $V_{i,i+j}:=q_{i+j-1}q_{i+j}^{-1}(V_{i,i+j-1})$, which is a subspace of $\mathcal E_{p_{i+j}}$, where $V_{i,i}=0$ is the $0$ subspace of $\mathcal E_{p_{i}}$.

Then $q_{i+j}^{-1}(V_{i,i+j})\cap q_{i+j+1}^{-1}(0)=0$.
\end{lema}
\begin{proof}
Follows from the definition of admissibility (\cite[Proposition 5]{DC}, \cite[Definition 3.2]{I}).
\end{proof}
\begin{lema}\label{Lem b}
Let $\mathcal{R}_T$ be a family of smooth curves of genus $0$ over a smooth variety $T$ i.e.,
$\pi:\mathcal{R}_T\rightarrow T$ is flat and every fibre is isomorphic to $\mathbb{P} ^1$.
Let $p_1$ and $p_2$ be two disjoint sections of $\pi$.
Let $\mathcal{E}$ be a vector bundle on $\mathcal{R}_T$ such that restriction to every fibre is Gieseker vector bundle.
Let $V_{_1}$ be a subbundle of $\mathcal E_{_{p_{_1}}}$
such that $q_{_1}^{-1}(V_1)\cap q_{_2}^{-1}(0)=(0)$. Then $q_{_2}(q_{_1}^{-1})(V_{_1})$ is a
subbundle of $\mathcal E_{_{p_{_2}}}$.
\end{lema}
\begin{proof}
Since $V_{_1}$ is a subbundle of $\mathcal E_{_{p_1}}$, $q_{_1}^{-1}(V_{_1})$ is a subbundle of $R^0\pi_*\mathcal E$.
From Lemma \ref{Int}, it follows that $q_{_1}^{-1}(V_{_1})\cap q_{_2}^{-1}(0)=0$.
So $q_1^{-1}(V_1)\oplus q_2^{-1}(0)$ is a subbundle of $R^0\pi_*\mathcal E$.
Therefore $q_2(q_1^{-1}(V_1)) $ is a subbundle of $\mathcal E_{p_2}$,
because $\frac{\mathcal E_{p_2}}{q_2(q_1^{-1}(V_1))}\cong \frac{R^0\pi_*\mathcal E}{q_1^{-1}(V_1)\oplus q_2^{-1}(0)}$,
which is of course a vector bundle.
\end{proof}
\begin{rema}\label{Inj}
To compute the rank of the subbundles $q_2(q_1^{-1}(V_1))$ and $q_1(q_2^{-1}(V_2))$ it is enough to compute them for each $t\in T$.
Therefore let us assume that we have a Gieseker vector bundle on $\mathbb{P}^1$
and we have two subspaces $V_1\subset \mathcal E_{p_1}$ and $V_2\subset \mathcal E_{p_2}$.
Then the ranks or dimension computation of $q_2(q_1^{-1}(V_1))$ and $q_1(q_2^{-1}(V_2))$ follows from \cite[Lemma 1, (ii)]{DC}.
For instance, if $\mathcal{E}\cong \mathcal{O}^{\oplus a }\oplus \mathcal{O}(1)^{\oplus b } $ and
$ V_{_{1 }}\subseteq \mathcal{O}_{_{p_1}}^{\oplus a }$ or $V_{_{2}}\subseteq \mathcal{O}_{_{p_2}}^{\oplus a }$
then the sub-spaces $q_2(q_1^{-1}(V_1))$ and $q_1(q_2^{-1}(V_2))$ are of rank $\text{dim}~V_{_1}+b$ or $\text{dim}~V_{_2}+b$
accordingly. More generally, if the projections
$ V_{_{1 }}\rightarrow \mathcal{O}_{_{p_1}}^{\oplus a }$ or $V_{_{2}}\rightarrow \mathcal{O}_{_{p_2}}^{\oplus a }$
are injective then the sub-spaces are of rank $\text{dim}~V_{_1}+b$ or $\text{dim}~V_{_2}+b$ accordingly.
\end{rema}

\begin{prop}\label{flag1}
Let $T$ be a smooth variety. Let $\mathcal X_T$ be a family of Gieseker curves such that
for every $t\in T$ the fibre $\mathcal X_{_{T,t}}\cong Y_n$ for some positive integer $n$.
Then given any Gieseker vector bundle $\mathcal E$ of rank $n$ over $\mathcal X_T$
there is a natural family of vector bundle over $X\times T$ with full-flag parabolic structures
along the two sections $p_1$ and $p_{n+1}$ of the projection morphism $X\times T\rightarrow T$.
\end{prop}
\begin{proof}
We have shown that $\mathcal X_T=X\times T\cup R_T$, where $R_T$ is a family of chains of $\mathbb{P}^1$'s parametrised by $T$ (Lemma \ref{Sec}). Given a Gieseker vector bundle $\mathcal E$ over $\mathcal X_T$, restrict the vector bundle on $X\times T$ and denote it by $V$.
Let us denote the restriction $\mathcal E|_{_{R_T}}$ by $W$. Then the vector bundle $\mathcal E$ is represented by a quadruple $(V, W, \phi_1, \phi_{n+1})$, where $\phi_1:V_{p_1}\xrightarrow{\cong} W_{p_1}, \phi_{n+1}:V_{p_{n+1}}\xrightarrow{\cong} W_{p_{n+1}}$. From Lemma \ref{Sec}, it follows that there are other $n-1$ number of
disjoint sections $p_2,\dots, p_{n}$ of the morphism $\mathcal{R}_T\rightarrow T$. Moreover, $R_T=\cup_{i=1}^{i=n}R_i$, where $R_i$ is a family of smooth rational curves over $T$ and $R_i\cap R_{i+1}=p_{i+1}$.

We will show that there are two full-flags of subbundles $F_{\bullet}$ and
$G_{\bullet}$ of the vector bundles $W_{p_1}$ and $W_{p_{n+1}}$ respectively. The inverse images of $\phi_1^{-1}(F_{\bullet})$ and
$\phi_{n+1}^{-1}(G_{\bullet})$ are the desired flags of $V_{p_1}$ and $V_{p_{n+1}}$ respectively.

Let us first construct the flag $G_{\bullet}$. Consider the following diagram :

\begin{equation}
\begin{tikzcd}
& R^0\pi_*W|_{_{R_{i}}}\arrow{dl}{q_{i}}\arrow{dr}{q_{i+1}}\\
W_{p_{i}} && W_{p_{i+1}}
\end{tikzcd}
\end{equation}

where $\pi$ is the restriction of the morphism $\pi: R_T\rightarrow T$ to $R_i$.

For $1\leq i\leq j\leq n$, we define $V_{i, j+1}:=q_{j+1}q_{j}^{-1}(V_{i,j})$ inductively with the initial condition that $V_{i,i}=$ the $0$ subbundle of $W_{p_i}$. Set $G_i:=V_{i,n+1}$. From Lemma \ref{Int}, \ref{Lem b} and Remark \ref{Inj},
it follows that $G_i$ is a subbundle of $W_{p_{n+1}}$ and $0\subset G_n\subset \dots\subset G_{2}\subset G_1=W_{p_{n+1}}$ is a full-flag of subbundles of $W_{p_{n+1}}$.

To construct the flag $F_{\bullet}$, for $1\leq j\leq i\leq n$, we define $V_{i, j}:=q_{j}q_{j+1}^{-1}(V_{i,j+1})$ inductively
by defining $V_{i,i}$ to be the $0$ subbundle of $W_{p_i}$. Set $F_i:=V_{n+2-i,1}$. From Lemma \ref{Int}, \ref{Lem b} and Remark \ref{Inj}, it follows that $F_i$ is a subbundle of $W_{p_{1}}$ and $0\subset F_n\subset \dots\subset F_{2}\subset F_1=W_{p_{1}}$ is a
full-flag of subbundles of $W_{p_{1}}$.
\end{proof}

\begin{rema}\label{IP}
In particular, the above proposition shows that given a rank $r$ Gieseker vector bundle $\mathcal E$ on $Y_n$ with $\mathcal E|_{R_i}=\mathcal{O}^{n-1}\oplus \mathcal{O}(1)$,
$i=1,\cdots n$, there is a canonical quasi-parabolic structure with full flags on $\tilde{E}:=\mathcal E|{_{X}}$ at the points $p_1$ and $p_{n+1}$.
\end{rema}

Now consider a triple $(V, F_{\bullet},G_{\bullet})$,
where $V$ is a vector bundle on the curve $X$, $F_{\bullet}$ is a full flag of $V_{p_1}$,
and $G_{\bullet}$ is a full-flag of $V_{p_{n+1}}$.
We say that such a triple is isomorphic to another such triple $(V', F'_{\bullet},G'_{\bullet})$
if there exists an isomorphism $\phi$ of $V$ and $V'$ which respects the flag structures i.e.,
these two triples are isomorphic as quasi-parabolic bundles.

\begin{lema}\label{mor}
If $(Y_{_n}, \mathcal{E})$ and $(Y_{_n}, \mathcal{E}')$ are equivalent (Definition \ref{EqGie}) as
Gieseker vector bundles then $(\widetilde{E}, F_{\bullet},G_{\bullet})$ and $(\widetilde{E}', F_{\bullet}',G_{\bullet}')$
are isomorphic, where $F_{\bullet}$, $G_{\bullet}$ (resp. $F_{\bullet}'$,$G_{\bullet}'$) are defined in Proposition \ref{flag1}.
\end{lema}

\begin{proof}
Consider the following commutative diagram:
\begin{equation}
\begin{tikzcd}
\mathcal E\arrow{r}{\phi}\arrow{d}& \mathcal E'\arrow{d}\\
Y_n \arrow{r}{\sigma}& Y_n
\end{tikzcd}
\end{equation}
Here $\phi$ is an isomorphism between $\mathcal E$ and $\mathcal E'$ and $\sigma$ is an automorphism of $Y_n$ fixing the sub curve $X$.

We restrict the diagram to $R_i$, the $i$-th $\mathbb P^1$:
\begin{equation}
\begin{tikzcd}
\mathcal E|_{R_i}\arrow{r}{\phi}\arrow{d}& \mathcal E|_{R_i}'\arrow{d}\\
R_i \arrow{r}{\sigma}& R_i
\end{tikzcd}
\end{equation}

Notice that the automorphism $\sigma$ fixes the two points $p_i$ and $p_{i+1}$.

Let us denote by $\tilde{\phi}$ the morphism $H^0(R_i, \mathcal E|_{R_i})\rightarrow H^0(R_i, \mathcal E'|_{R_i})$ induced by the morphism $\phi$. We also denote by $\phi_{p_i}$ the evaluation of the morphism $\phi$ at the point $p_i$ for $i=1,\dots, n+1$.

We have the following commutative diagram:

\begin{equation}
\begin{tikzcd}
& H^0(R_{i}, \mathcal{E}|_{_{R_{i}}})\arrow{dl}{q_{i}}\arrow{dddr}{q_{i+1}} \arrow{rrrr}{\tilde{\phi}} &&&& H^0(R_{i}, \mathcal{E}'|_{_{R_{i}}})\arrow{dl}{q'_{i}}\arrow{dddr}{q'_{i+1}}\\
\mathcal{E}_{p_{i}}\arrow{rrrr}{\phi_{p_i}} &&&& \mathcal{E}'_{p_{i}}\\ \\
&&\mathcal{E}_{p_{i+1}}\arrow{rrrr}{\phi_{p_{i+1}}} &&&& \mathcal{E}'_{p_{i+1}}
\end{tikzcd}
\end{equation}

Notice that if $V$ is any subspace of $\mathcal E_{p_i}$ then we have an induced isomorphism
\begin{equation}\label{flag111}
\phi_{p_{i+1}}: q_{i+1}q_i^{-1}(V)\rightarrow q_{i+1}'q_i'^{-1}(\phi_{p_i}(V))
\end{equation}

For $1\leq i\leq j\leq n$ let $V_{{i, j+1}}$ denote the intermediate vector spaces constructed in Proposition \ref{flag1} for the Gieseker vector bundle $\mathcal E$ and let $V'_{{i, j+1}}$ denote the intermediate vector spaces constructed in Proposition \ref{flag1} for the Gieseker vector bundle $\mathcal E'$.

Notice in the construction of the flag in proposition \ref{flag1}, $V_{_{i,i}}\xrightarrow{\phi_{_{p_{i}}}} V'_{_{i,i}}$ is an isomorphism, because $V_{_{i,i}}$ and $V'_{i,i}$ are both zero vector spaces. Therefore from the equation \ref{flag111}, it follows that $V_{_{i,i+1}}\xrightarrow{\phi_{_{p_{i+1}}}} V'_{_{i,i+1}}$ is an isomorphism. From this it follows that $V_{_{i,i+2}}\xrightarrow{\phi_{_{p_{i+2}}}} V'_{_{i,i+2}}$ is an isomorphism and iterating in this manner finally we see that $V_{_{i,n+1}}\xrightarrow{\phi_{_{p_{n+1}}}} V'_{_{i,n+1}}$ is an isomorphism. Therefore the flag (constructed in proposition \ref{flag1}) at the point $p_{n+1}$ corresponding to $\mathcal E|_{_{R}}$ is mapped to the flag at the point $p_{n+1}$ corresponding to $\mathcal E'|_{_{R}}$ by the isomorphism $\phi_{p_{n+1}}$.

Similarly the flag (constructed in proposition \ref{flag1}) at the point $p_{1}$ corresponding to $\mathcal E|_{_{R}}$ is mapped isomorphically to the flag at the point $p_{1}$ corresponding to $\mathcal E'|_{_{R}}$ under the isomorphism $\phi_{p_1}$.

Notice that the induced isomorphism between $\tilde{E}$ and $\tilde{E}'$ naturally agrees with the isomorphism
of $\mathcal{E}|_{_R}$ and $\mathcal{E}'|_{_R}$ at the meeting points $p_{_1}$ and $p_{_{n+1}}$. Therefore from the equation \ref{flag1} it follows that if two Gieseker bundles are equivalent then the induced parabolic triples are also isomorphic.
\end{proof}

\section{Automorphism of Gieseker vector bundles:}

\begin{lema}\label{Uni}
Up to isomorphism there is only one strictly standard rank $n$ vector bundle $\mathcal{E}$ on a chain $R$ of length $n$.
\end{lema}
\begin{proof}
We know from \cite[Lemma 5.4]{VPD} that $\mathcal{E}=\oplus_{i=1}^{n} \mathcal{L}_{_i}$,
where $\mathcal{L}_{_i}$ is $\mathcal{O}(1)$ on $R_{_i}$ and $\mathcal{O}$ elsewhere.
If $\mathcal{E}'$ is another strictly standard rank $n$ vector bundle on $R$ then $\mathcal{E}'=\oplus_{i=1}^{n} \mathcal{L}'_{_i}$,
where $\mathcal{L}'_{_i}$ is $\mathcal{O}(1)$ on $R_{_i}$ and $\mathcal{O}$ elsewhere.
It is easy to check that $\mathcal{L}_{_i}\cong \mathcal{L}'_{_i}$. Therefore $\mathcal{E}\cong \mathcal{E}'$.
\end{proof}

\begin{rema}\label{Iden}
The lemma \ref{Uni} implies that, for simplicity, we can fix one such bundle on the chain $R$ and
we can also assume that all connecting fibre identifications are some fixed elements in the tori
with respect to the basis corresponding to the direct sum decomposition i.e., in the
tori $\times_{i=1}^{i=n} Isom(\mathcal{L}_i|_{R_j,p_{j+1}}, \mathcal{L}_i|_{R_{j+1},p_{j+1}})$ for all $1\leq j\leq n-1$.
\end{rema}

\begin{lema}\label{3.8}
Let $\mathcal{E}$ be a rank $n$ Gieseker bundle on the chain $R$ of $\mathbb{P}^1$s of length $n$.
Let $p_1$ and $p_{n+1}$ be two points on the extremal irreducible components which are also smooth points of $R$.
We decompose $\mathcal{E}=\mathcal{L}_{_1}\oplus \dots \oplus \mathcal{L}_{_{n}}$.
We consider two copies of $GL_n$ namely $Aut(\mathcal{E}_{_{p_1}})$ and $Aut(\mathcal{E}_{_{p_{n+1}}})$ and fix a maximal torus $T$ namely $\times_{_{i=1}}^{n} Aut(\mathcal{L}_{_{i}})$. Let us consider the following flag at $p_1$:
$$
\mathcal E_{_{p_1}}\supset \mathcal L_{1,p_1}\oplus\dots\oplus\mathcal L_{n-1,p_1}\supset\dots\supset \mathcal L_{1,p_1}\supset 0
$$
and the following flag at $p_{n+1}$:
$$
\mathcal E_{_{p_{n+1}}}\supset \mathcal L_{2,p_{n+1}}\oplus\dots\oplus\mathcal L_{n,p_{n+1}}\supset\dots\supset \mathcal L_{n,p_{n+1}}\supset 0
$$
We define $B^+$ to be the Borel subgroup of $Aut(\mathcal E_{p_1})$ to be the stabilizer of the flag defined above. Similarly, we define $B^-$ to be the Borel subgroup of $Aut(\mathcal E_{p_{n+1}})$ to be the stabilizer of the flag defined above.

Then the morphism given by the evaluation of the automorphism at
the two extremal points $Aut(\mathcal{E})\rightarrow Aut(\mathcal{E}_{p_1})\times Aut(\mathcal{E}_{_{p_{n+1}}})$ is injective.
There is a natural projection $q:B^+\times B^-\rightarrow T\times T$. The image of the morphism
$Aut(\mathcal{E})\rightarrow Aut(\mathcal{E}_{p_1})\times Aut(\mathcal{E}_{_{p_{n+1}}})$ is $q^{-1}(\Delta)$,
where $\Delta$ is the diagonal of $T\times T$.
\end{lema}
\begin{proof}
We know from \cite[Lemma 5.5]{VPD}, $Hom(\mathcal{E},\mathcal E)=\oplus_{i,j} Hom(\mathcal{L}_{_i}, \mathcal{L}_{_j})$.
Therefore an element $\sigma\in Aut(\mathcal{E})$ is a matrix $(\sigma_{_{i,j}})$,
where $\sigma_{_{i,j}}\in Hom(\mathcal{L}_{_i}, \mathcal{L}_{_j})$. Since $Hom(\mathcal{L}_{_i}, \mathcal{L}_{_i})=\mathbb{C}$,
we know that $\sigma_{_{ii}}$'s are scalars. Since the evaluations $\sigma_{_{p_1}}$ and
$\sigma_{_{p_{n+1}}}$ fix the flags $F_{\bullet}$ and $G_{\bullet}$ respectively,
$\sigma_{_{p_1}}\in B^+$ and $\sigma_{_{p_{n+1}}}\in B^-$.
Since $\sigma_{_{ii}}$'s are scalars we conclude that the image lies in $q^{-1}(\Delta)$,
where $\Delta$ is the diagonal of the maximal torus $T\times T$ of $B^+\times B^-$ ($T$ is described in the statement of the lemma).
Using the description of the $End(\mathcal{E})$,
we see that $Aut(\mathcal{E})\rightarrow Aut(\mathcal{E}_{p_1})\times Aut(\mathcal{E}_{_{p_{n+1}}})$ is injective.

Now let $u^+\in U^+$ and $u^-\in U^-$ and $t\in T$. We write $u^+$ as $(u^+_{ij})$ and $u^-$ as $(u^-_{ij})$ and $t$ as $(t_{ij})$.
We want to find a $\sigma:=(\sigma_{ij})\in Aut(\mathcal{E})$ such that the two restrictions at the points $p_1$ and $p_{n+1}$ are
$b^+:=t\cdot u^+$ and $b^-:=t\cdot u^-$. We write $b^+$ as $(b^+_{ij})$ and $b^-$ as $(b^-_{ij})$.
Since $Hom(\mathcal{L}_{_i}, \mathcal{L}_{_j})=\mathbb{C}$, $\sigma_{ij}$ is determined by $b^{+}_{ij}$ if $i<j$ and
$\sigma_{ij}$ is determined by $b^{-}_{ij}$ if $i>j$. Now the proof follows from the Corollary \ref{3.10}.
\end{proof}

\begin{lema}
Let $C$ be any connected projective curve and $p\in C$ be any point. Let $\mathcal{E}$ be a
vector bundle on $C$ and $\sigma\in End (\mathcal{E})$.
Suppose $\sigma(p)\in Aut(\mathcal{E}_p)$. Then $\sigma\in Aut(\mathcal{E})$.
\end{lema}
\begin{proof}
Given $\sigma:\mathcal{E}\rightarrow \mathcal{E}$,
we consider $\text{det} ~\sigma: \text{det}~\mathcal{E}\rightarrow \text{det}~\mathcal{E}$.
We note that $\text{det}~\sigma\in H^0(C, \mathcal{O}_{_C})=\mathbb{C}$.
Therefore if $\text{det} ~\sigma(p)\neq 0$, $\text{det} ~\sigma(p)\neq 0$ $\forall p\in C$.
Therefore $\sigma\in Aut(\mathcal{E})$.
\end{proof}

\begin{coro}\label{3.10}
Let $\mathcal E$ be a Gieseker vector bundle and $\sigma\in End(\mathcal{E})$ such that
$\sigma_{_{p_1}}\in Aut(\mathcal{E}_{p_1})$ and $\sigma_{_{p_{n+1}}}\in Aut(\mathcal{E}_{p_{n+1}})$.
Then $\sigma\in Aut(\mathcal{E})$.
\end{coro}

\subsection{Equivariant $\mathbb C^*$ action on the line bundle $\mathcal O(1)$ over $\mathbb P^1$: }

We know that $\mathbb P^1=\frac{\mathbb C^2\setminus(0,0)}{\mathbb C^*}$. Let us denote by $N$ the point $[1:0]$ and by $S$ the point $[0:1]$. Let us call $N$ the north pole and $S$ the south pole. The automorphism group of $\mathbb P^1$ fixing the two poles is
\[
\Bigg\{ \begin{pmatrix}
\frac{1}{\lambda} & 0 \\
0 & \lambda
\end{pmatrix}| \lambda\in \mathbb C^*\Bigg\}=\mathbb G_m.
\]
So each element \[
\begin{pmatrix}
\frac{1}{\lambda} & 0 \\
0 & \lambda
\end{pmatrix}
\] induces an automorphism $\phi_{\lambda}$ of $\mathbb P^1$. Moreover given any line bundle $L$ we see that $\phi_{\lambda}^*L:=L\times_{\mathbb{P}^1} \mathbb P^1 \cong L$. This induces two isomorphisms $L_{N}\rightarrow L_N$ and $L_S\rightarrow L_S$, where $L_N$ and $L_S$ denotes the fibres of $L$ at the points $N$ and $S$ respectively. If $L$ is the trivial bundle then these isomorphisms are, in fact, identity morphisms.

\begin{equation}
\begin{tikzcd}
\mathbb{P}^1\times \mathbb C \arrow{r}\arrow{d} & \mathbb{P}^1\times \mathbb C \arrow{d} && ([a:b], v)\mapsto ([\frac{1}{\lambda} a:\lambda b], v)\\
\mathbb P^1\arrow{r} & \mathbb{P}^1 && {[a:b]}\mapsto [\frac{1}{\lambda} a:\lambda b]
\end{tikzcd}
\end{equation}

If $L$ is the total space of $\mathcal O(-1)$ then we see that $L_N\rightarrow L_N$ is multiplication by $\frac{1}{\lambda}$ and $L_S\rightarrow L_S$ is multiplication by $\lambda$.

\begin{equation}
\begin{tikzcd}
L \arrow{r}\arrow{d} & L \arrow{d} && ([a:b], (\gamma a,\gamma b))\mapsto ([\frac{1}{\lambda} a:\lambda b], (\frac{1}{\lambda} \gamma a,\lambda \gamma b))\\
\mathbb P^1\arrow{r} & \mathbb{P}^1 && {[a:b]}\mapsto [\frac{1}{\lambda} a:\lambda b]
\end{tikzcd}
\end{equation}

Similarly, If $L$ is the total space of $\mathcal O(1)$ then we see that $L_N\rightarrow L_N$ is multiplication by
$\lambda$ and $L_S\rightarrow L_S$ is multiplication by $\frac{1}{\lambda}$.

\textbf{Convention: We fix a convention that for $\lambda\in \mathbb G_m$
the two induced isomorphisms $\mathcal O(1)_N\rightarrow \mathcal O(1)_N$, $\mathcal O(1)_S\rightarrow \mathcal O(1)_S$
are multiplications by $\lambda$ and $\frac{1}{\lambda}$ respectively.}
\subsection{Equivariant torus action on a Gieseker vector bundle:}\label{Act}

Let $Y_n$ be a Gieseker curve (Definition \ref{gis}),
which is union of $X$ and $R$, $R$ is a chain of $\mathbb{P}^1$s of length $n$.
Let us denote the nodes by $p_{_1},\dots,p_{_{n+1}}$. The extremal nodes on the chain $R$ are $p_1$ and $p_{_{n+1}}$.

The group $Aut(R, p_1, p_{n+1})$ of automorphisms of $R$ fixing the two points $\{p_1, p_{n+1}\}$ is isomorphic to $\times_{i=1}^n \mathbb{G}_m$.
The $i$-th copy of $\mathbb G_m$ is the automorphism of $R_i$ fixing the points $\{p_i, p_{i+1}\}$.
Let us consider an element $(\lambda_1,\dots,\lambda_n)\in \times_{i=1}^n \mathbb{G}_m$.
We declare, without loss of generality, that $p_i$ is north-pole (denote by N) of $R_{i-1}$ and
$R_i$ if $i$ is odd else we declare it south-pole (denote by S). It looks like the following picture:
\\ \\ \\ \\
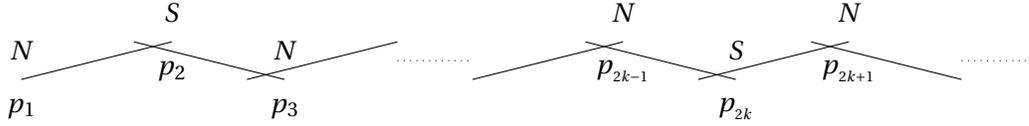
\begin{tikzpicture}[overlay, scale=1, xshift= 1cm]
\draw node[yshift=-2ex]{$p_1$} node[yshift=2ex]{$\tiny{N}$} (0,0) -- (2,.5) node[yshift=2ex]{$S$}node[yshift=-2ex]{$p_2$}; \draw (1.5,.5) -- (3.5,0)node[yshift=-2ex]{$p_3$} node[yshift=2ex]{$\tiny{N}$}; \draw (3,0) -- (5,.5); \draw[dotted] (5,.25) -- (6,.25);\draw (6,0) -- (8,.5)node[yshift=-2ex]{$p_{_{2k-1}}$}node[yshift=2ex]{$\tiny{N}$}; \draw (7.5,.5) -- (9.5,0)node[yshift=-2ex]{$p_{_{2k}}$}node[yshift=2ex]{$\tiny{S}$}; \draw (9,0) -- (11,.5) node[yshift=-2ex]{$p_{_{2k+1}}$}node[yshift=2ex]{$\tiny{N}$}; \draw (10.5,.5) -- (12.5,0) ; \draw[dotted] (12.5,.25) -- (13.5,.25);
\end{tikzpicture}
\\ \\ \\ \\

Let us consider a Gieseker vector bundle $\mathcal E$ of rank $n$ over $Y_n$.
We can express it as a unique tuple $(V, W_1, \dots, W_n, \phi_1,\dots ,\phi_{n+1})$,
where $V$ is a vector bundle on $X, W_i$ is the vector bundle $\mathcal O^{n-1}\oplus \mathcal O(1)$
and $\phi_1: V_{p_1}\xrightarrow{\cong} W_{_{1,p_1}}$, $\phi_{n+1}:V_{p_{n+1}}\xrightarrow{\cong} W_{_{n,p_{n+1}}}$ and
$\phi_i: W_{i-1,p_{i}}\xrightarrow{\cong} W_{i,p_{i}}$ for $2\leq i\leq n$.

Let $Aut(Y_n/Y)$ denote the group of automorphisms of $Y_n$, which commute with the projection morphism to $Y$. Notice that $Aut(Y_n/Y)$ is also the subgroup of $Aut(Y_n)$, which consists of all the automorphisms which are identity on the sub curve $X$. Let us consider an element $\sigma\in Aut(Y_n/Y)$. We want to describe the Gieseker vector bundle $\sigma^*\mathcal E$ as a tuple.

Let us denote by $D_{_{i_1,\dots, i_k}}(\lambda_1,\dots,\lambda_k)$ the diagonal matrix whose $i_j$-th
diagonal element is $\lambda_j$ and $1$ everywhere else.

Then $\sigma^*(V, W_1, \dots, W_n, \phi_1,\dots ,\phi_{n+1})=(V, W_1, \dots, W_n, D_{1}(\lambda_1)\cdot \phi_1, D_{2}(\lambda_2^{-1})\cdot \phi_2\cdot D_{1}(\lambda_1), D_{3}(\lambda_3)\cdot \phi_3\cdot D_{2}(\lambda_2^{-1}),\dots, D_{2k}(\lambda_{2k}^{-1})\cdot \phi_{2k}\cdot D_{2k-1}(\lambda_{2k-1}),D_{2k+1}(\lambda_{2k+1})\cdot \phi_{2k+1}\cdot D_{2k}(\lambda_{2k}^{-1}), \dots, D_{n-1}(\lambda^{(-1)^{n}}_{n-1})\cdot \phi_{n-1}\cdot D_{n-2}(\lambda^{(-1)^{n-1}}_{n-2}), D_n(\lambda^{(-1)^{n+1}}_n)\cdot \phi_{n+1} )$,
where $\sigma|_{_R}=(\lambda_1,\cdots, \lambda_n)\in Aut(R, p_{_1}, p_{_{n+1}})$, the group of automorphisms of $R$ which fix the two points $p_{_1}$ and $p_{_{n+1}}$.

\begin{lema}\label{Action}
Let $\mathcal E$ be a rank $n$ Gieseker vector bundle on $Y_n$. Let $(V, W_1, \dots, W_n, \phi_1,\dots ,\phi_{n+1})$ be
the unique tuple representing $\mathcal E$, where $V$ is a vector bundle
on $X, W_i$ is the vector bundle $\mathcal O^{\oplus n+1}\oplus \mathcal O(1)$ and
$\phi_1: V_{p_1}\xrightarrow{\cong} W_{_{1,p_1}}$,
$\phi_{n+1}:V_{p_{n+1}}\xrightarrow{\cong} W_{_{n,p_{n+1}}}$ and
$\phi_i: W_{i-1,p_{i}}\xrightarrow{\cong} W_{i,p_{i}}$ for $2\leq i\leq n$.
Let us consider an element $\sigma \in Aut(Y_n/Y)$. Then
$\sigma^*\mathcal E\cong (V, W_1, \dots, W_n, \phi_1, \phi_2, \phi_3,\dots, \phi_{2k}, \dots, \phi_{n}, D_{1,2,\dots,n}(\lambda_1^2, \lambda_2^{-2},\dots, \lambda_n^{(-1)^{(n+1)}2})\phi_{n+1})$,
where $(\lambda_1,\cdots, \lambda_n)\in Aut(R, p_{_1}, p_{_{n+1}})$.
\end{lema}
\begin{proof}
We can assume that the identifications $\phi_2,\dots,\phi_{n}$ are diagonal matrices with respect to a choice of basis
whose elements are elements of the corresponding fibres of $\mathcal{L}_{_{i,p_j}}$ (Lemma \ref{Uni} and Remark \ref{Iden}).
Since $\phi_i$'s are all elements of the torus they commute with $D_{i}$'s and therefore

$(\lambda_1,\dots,\lambda_n)^*(V, W_1, \dots, W_n, \phi_1,\dots ,\phi_{n+1})\hspace{-1cm}=\hspace{-1cm}(V, W_1, \dots, W_n, D_{1}(\lambda_1)\hspace{-1cm}\cdot \phi_1, D_{1,2}(\lambda_1,\lambda_2^{-1})\cdot \phi_2, D_{2,3}(\lambda_2^{-1},\lambda_3)\cdot \phi_3,\dots, D_{2k-1,2k}(\lambda_{2k-1},\lambda_{2k}^{-1})\cdot \phi_{2k},D_{2k,2k+1}(\lambda_{2k}^{-1},\lambda_{2k+1})\cdot \phi_{2k+1},\dots, D_{n-2,n-1}(\lambda^{(-1)^{n-1}}_{n-2},\lambda^{(-1)^n}_{n-1})\phi_{n}, D_n(\lambda^{(-1)^{n+1}}_n)\phi_{n+1})$.

The automorphism group of the vector bundle $\mathcal{O}^{n-1}\oplus \mathcal{O}(1)$ also contains the obvious torus $\prod_{i=1}^{n-1} Aut(\mathcal O)\times Aut(\mathcal O(1))\cong \times_{i=1}^n \mathbb{G}_m$. Therefore all $D_{_{i_1,\dots, i_k}}(\lambda_1,\dots,\lambda_k)$'s are in this automorphism group of $\mathcal{O}^{n-1}\oplus \mathcal{O}(1)$. Therefore

$
(\lambda_1,\dots,\lambda_n)^*(V, W_1, \dots, W_n, \phi_1,\dots ,\phi_{n+1})\hspace{-1cm}\cong \hspace{-1cm}(V, W_1, \dots, W_n, \phi_1, D_{1,2}(\lambda_1^2,\lambda_2^{-1})\cdot \phi_2, D_{2,3}(\lambda_2^{-1},\lambda_3)
\cdot \phi_3,\dots, D_{2k-1,2k}(\lambda_{2k-1},\lambda_{2k}^{-1})\cdot \phi_{2k},D_{2k,2k+1}(\lambda_{2k}^{-1},\lambda_{2k+1})\cdot \phi_{2k+1},\dots,D_{n-2,n-1}(\lambda^{(-1)^{n-1}}_{n-2},\lambda^{(-1)^n}_{n-1})\phi_{n},D_n(\lambda^{(-1)^{n+1}}_n)\phi_{n+1})
$

$\cong (V, W_1, \dots, W_n, \phi_1, \phi_2, D_{1,2,3}(\lambda_1^{2}, \lambda_2^{-2},\lambda_3)\cdot \phi_3,\dots, D_{2k-1,2k}(\lambda_{2k-1},\lambda_{2k}^{-1})\cdot \phi_{2k},D_{2k,2k+1}(\lambda_{2k}^{-1},\lambda_{2k+1})\cdot \phi_{2k+1},\dots,D_{n-2,n-1}(\lambda^{(-1)^{n-1}}_{n-2},\lambda^{(-1)^n}_{n-1})\phi_{n}, D_n(\lambda^{(-1)^{n+1}}_n)\phi_{n+1}))$

$\dots$

$\cong (V, W_1, \dots, W_n, \phi_1, \phi_2, \phi_3,\dots, \phi_{2k}, \dots, \phi_{n}, D_{1,2,\dots,n}(\lambda_1^2, \lambda_2^{-2},\dots, \lambda_n^{(-1)^{(n+1)}2})\phi_{n+1})$.
\end{proof}

\section{Description of the most singular loci}

\begin{lema}\label{Deg} Let $Y$ be an irreducible nodal curve with only one node and $\pi:X\rightarrow Y$ be the normalization. Let us denote the node by $y$ and its preimages by $\{y_1,y_2\}$.
\begin{enumerate}
\item Let $L$ be a vector bundle of rank $r$ on $X$. Then deg $\pi_*L=$ deg $L+r$.
\item Let $\mathcal L$ be a torsion free sheaf of rank $r$ on $Y$ of local type $\widehat{\mathcal L}_y\cong \widehat{\mathcal O}^{\oplus a}_y\oplus \widehat{m}^{\oplus b}_y$, where $m_y$ denotes the maximal ideal of $y$. Then deg $\frac{\pi^*\mathcal L}{\tt{torsion}}=$deg $\mathcal L-b$.
\end{enumerate}
\end{lema}
\begin{proof}
Since the normalization map is affine, we have deg $\pi_*L=\chi(\pi_*L)-r\chi(\mathcal O_Y)=\chi(L)-r(\chi(\mathcal O_X)-1)=(\chi(L)-r\chi(\mathcal O_X))+r=$deg $L+r$.

For the second assertion, consider the following short exact sequence:
\begin{equation}
0\rightarrow \mathcal L\rightarrow \pi_*\big(\frac{\pi^*\mathcal L}{\tt{torsion}}\big)\rightarrow Q\rightarrow 0,
\end{equation}
where $Q$ is a torsion supported only at the node $y$.
It is enough to compute the dimension of $Q$ and for that purpose, it is enough to consider the local-analytic picture of the above short-exact sequence. Notice that the natural map

\begin{equation}
\widehat{m}_y\xrightarrow{\cong} \pi_*\big(\frac{\pi^*\widehat{m}_y}{\tt{torsion}}\big)
\end{equation}
is an isomorphism and that we have a short exact sequence
\begin{equation}
0\rightarrow \widehat{\mathcal O}_y\rightarrow \pi_*\big(\frac{ \pi^*\widehat{\mathcal O}_y}{\tt{torsion}}\big)\rightarrow k(y)\rightarrow 0.
\end{equation}
From these two local calculations it follows that dim $Q=a$. Now using $(1)$ we conclude that
$$
\text{deg}~~ \frac{\pi^*\mathcal L}{\tt{torsion}}=\text{deg}~~ \mathcal L-b.
$$
\end{proof}
\begin{lema}\label{Stab}
A rank $n$ Gieseker vector bundle $(Y_{_n}, \mathcal E)$ is stable if and only if $\widetilde{\mathcal E}:=\mathcal E|_{X}$ is
stable vector bundle of rank $n$ and degree $d-n$.
\end{lema}
\begin{proof}

Consider the following exact sequence of $\mathcal{O}_{_{Y_n}}$-modules:

\begin{equation}
0\rightarrow \mathcal{E}|_{R}(-p_{_1}-p_{_{n+1}})\rightarrow \mathcal{E}\rightarrow \widetilde{\mathcal E}\rightarrow 0
\end{equation}

Let us denote the morphism $Y_{_n}\rightarrow Y$ by $\pi_{_n}$. Notice that $\pi_n|_{_X}=\pi$.

We have the following
exact sequence of $\mathcal{O}_{_{Y}}$-modules:

\begin{equation}
0\rightarrow \pi_{_{n*}}\mathcal{E}|_{R}(-p_{_1}-p_{_{n+1}})\rightarrow \pi_{_{n*}}\mathcal{E}\rightarrow \pi_{_{n*}}\widetilde{\mathcal E}\rightarrow R^{^1} \pi_{_{n*}} \mathcal{E}|_{R}(-p_{_1}-p_{_{n+1}})\dots
\end{equation}

Now it is straightforward to check that $\pi_{_{n*}}\mathcal{E}|_{R}(-p_{_1}-p_{_{n+1}})=0$ and
$R^{^1} \pi_{_{n*}} \mathcal{E}|_{R}(-p_{_1}-p_{_{n+1}})=0$.
Therefore we have $\pi_{_{n*}} \mathcal{E}\cong \pi_* \widetilde{\mathcal E}$.
Let us denote this torsion free sheaf by $\mathcal{F}$.

By definition, a Gieseker vector bundle is stable if and only if the corresponding torsion free sheaf is stable ( Definition \ref{stabilityDEF}).
So it is enough to prove that $\mathcal{F}$ is stable if and only if $\widetilde{\mathcal E}$ is stable. Notice degree of $\widetilde{\mathcal E}$ is $d-n$ and $(n,d-n)=1$.

Also note that we have an injective morphism $\frac{\pi^*\mathcal{F}}{Torsion}\hookrightarrow \widetilde{\mathcal E}$.
By degree calculation, we see that their degrees are same, therefore $\frac{\pi^*\mathcal{F}}{Torsion}\cong \widetilde{\mathcal E}$.

For any subbundle $L$ of $\widetilde{\mathcal E}$, we get a saturated torsion free sub-sheaf $\pi_*L$ of $\mathcal{F}$.
Given a saturated torsion-free subsheaf $\mathcal{L}$, we have a bundle $L':=\frac{\pi^*\mathcal{L}}{Torsion}$
which is a subsheaf of $\widetilde{\mathcal E}$. If the torsion-free subsheaf $\mathcal{L}$ is of local-type $\mathcal{O}^{\oplus a}\oplus m^{\oplus b}$,
then $\text{deg}~ L'=\text{deg}~\mathcal{L}-b$ (statement (2) of Lemma \ref{Deg}).

Suppose now that $\mathcal{F}$ is stable. Let $L$ be a subbundle of $\widetilde{\mathcal E}$. Then $\pi_* L$ is a saturated subsheaf of $\mathcal{F}$.
So we have
$\frac{deg \pi_{*} L}{rk (L)}< \frac{deg \mathcal{F}}{rk (\mathcal{F})}\implies \frac{deg L+rk (L)}{rk (L)}< \frac{deg \widetilde{\mathcal E}+n}{n}\implies \frac{deg L}{rk(L)}< \frac{deg \widetilde{\mathcal E}}{n}$.

Now suppose $\widetilde{\mathcal E}$ is stable. Let $\mathcal{L}$ be a saturated subsheaf of $\mathcal F$ of local-type $\mathcal{O}^{\oplus a}\oplus m^{\oplus b}$,
where $a+b=rk(\mathcal{L})$.
Then we have
$\frac{deg L'}{rk (L')}< \frac{deg \widetilde{\mathcal E}}{n}\implies \frac{deg \mathcal{L}-b}{rk (\mathcal{L})}< \frac{deg \widetilde{\mathcal E}}{n}\implies \frac{deg \mathcal{L}}{rk(\mathcal{L})}< \frac{deg \widetilde{\mathcal E}}{n}+\frac{b}{rk(\mathcal{L})}= \frac{deg \widetilde{\mathcal E}\cdot rk(\mathcal{L})+bn}{n\cdot rk(\mathcal{L})}\leq \frac{rk(\mathcal{L})\cdot (deg \widetilde{\mathcal E}+n)}{n\cdot rk(\mathcal{L})}=\frac{deg \mathcal{F}}{n}$.
Therefore $\mathcal{F}$ is stable.
\end{proof}

\begin{rema}\label{parabolic weights}
From Lemma \ref{Stab}, it follows that if the Gieseker bundle $\mathcal E$ is stable then $\tilde{E}$ is also stable. Thus there exists a pair of sets of small rational weights $0<\alpha_{i}^{p_1}<1$ and $0<\beta_{i}^{p_{n+1}}<1$ $(\text{for}~~i=1,\dots, n)$ such that $(\tilde{E}, F_{\bullet}, G_{\bullet})$ is a stable parabolic bundle for any quasiparabolic structures $F_{\bullet}$ and $G_{\bullet}$ with respect to the weights $\{\alpha_{i}^{p_1}\}$ and $\{\beta_{i}^{p_{n+1}}\}$. So, in particular, $\tilde{E}$ with the parabolic structure induced by the Gieseker bundle $\mathcal E$ (remark \ref{IP}) is also stable( see \cite[Proposition 2.6]{DT I}).
\end{rema}

Let $\mathcal{M}_{_{\tt{VB}_{n,d-n}}}$ be the moduli space of stable vector bundles of rank $n$ and
degree $d-n$ on $X$.
Let $\mathcal{U}$ be a universal bundle over $X\times \mathcal{M}_{_{\tt{VB}_{n,d-n}}}$ and
$\mathcal{U}_{_{p_i}}:=\mathcal{U}|_{p_i\times \mathcal{M}_{_{\tt{VB}_{n,d-n}}}}$, $i=1,n+1$,
where $p_1:=X\cap R_1$ and $p_{n+1}:=X\cap R_{n}$. Let $FL_n(\mathcal{U}_{_{p_i}})$ be the
flag bundles whose fibres are flag varieties of complete flags. By Remark \ref{parabolic weights}, there exists small sets of weights $\{\alpha_{i}^{p_1}\}$
and $\{\beta_{i}^{p_{n+1}}\}$
such that the variety
$FL_n(\mathcal{U}_{_{p_1}})\times_{_{\mathcal{M}_{_{\tt{VB}_{n,d-n}}}}} FL_n(\mathcal{U}_{_{p_{n+1}}})$ is isomorphic to the moduli space
of stable parabolic bundles over $X$ with parabolic structures at $p_1$ and $p_{n+1}$ with complete flag types and parabolic weights $\{\alpha_{i}^{p_1}\}$
and $\{\beta_{i}^{p_{n+1}}\}$.

\begin{thm}\label{main1}
There is a natural isomorphism $f:\mathcal{M}^n\to FL_n(\mathcal{U}_{_{p_1}})\times_{_{\mathcal{M}_{_{\tt{VB}_{n,d-n}}}}} FL_n(\mathcal{U}_{_{p_2}})$.
\begin{proof}

Let $\mathcal{W}\to \mathcal{Y}^{^{\tt{st}}}$ be the universal Gieseker curve and $\mathcal{E}\to \mathcal{W}$ a universal Gieseker bundle. Let us denote by $\mathcal W^n$ the restriction of $\mathcal W$ to the closed subscheme $\mathcal Y^{n,\tt{st}}$. We denote the restriction of $\mathcal{E}$ on $\mathcal{W}^n$, the schematic pre-image of
$\mathcal{Y}^{n,\tt{st}}$ also by $\mathcal{E}$. Then by the proof of Proposition \ref{Strat}, $\mathcal{Y}^{n,\tt{st}}$ is a smooth variety.
Thus, by using Proposition \ref{flag1}, we get
$(\widetilde{\mathcal{E}}:=\mathcal{E}|_{_{X\times \mathcal{Y}^{n,\tt{st}}}},F_{\bullet},G_{\bullet})$, a
family of parabolic bundles parametrised by $\mathcal{Y}^{n,\tt{st}}$. From Lemma \ref{Stab} and Remark \ref{parabolic weights}, it follows that it is, in fact, a family of stable parabolic bundles. Therefore, by the universal property of
the moduli space of stable parabolic bundles we get a morphism:
\[\mathcal{Y}^{n,\tt{st}}\to FL_n(\mathcal{U}_{_{p_1}})\times_{_{\mathcal{M}_{_{VB_{n,d-n}}}}} FL_n(\mathcal{U}_{_{p_2}}).\]
Using Lemma \ref{mor}, we see that this morphism descends to a morphism-
\[f:\mathcal{M}^n\to FL_n(\mathcal{U}_{_{p_1}})\times_{_{\mathcal{M}_{_{VB_{n,d-n}}}}} FL_n(\mathcal{U}_{_{p_2}}).\]

Now note that the dimension of $FL_n(\mathcal{U}_{_{p_1}})\times_{_{\mathcal{M}_{_{VB_{n,d-n}}}}} FL_n(\mathcal{U}_{_{p_2}})$ is
$n^2(g-2)+1+n(n-1)=n^2(g-1)+1-n=\mbox{dim}(\mathcal{M}^{n})$. Since both the varieties are smooth and projective and have
same dimension if we can show that $f$ is injective then by Zariski's main theorem it follows that $f$ is an
isomorphism. The injectivity of $f$ will follow from Proposition \ref{Immersion}.
\end{proof}
\end{thm}

\begin{prop}\label{Immersion}
Let $\mathcal{E}$ and $\mathcal{E}'$ be two Gieseker vector bundles of rank $n$ on $Y_n$ such that the
corresponding parabolic triples $(V, F_{\bullet}, G_{\bullet})$ and $(V', F'_{\bullet}, G'_{\bullet})$, constructed in Proposition \ref{flag1}
are isomorphic.
Then the Gieseker bundles are equivalent.
\end{prop}

\begin{proof}
Let us write the Gieseker bundle $\mathcal{E}$ as a tuple $(V, W, \phi_1,\phi_{n+1})$
where $V$ is a vector bundle on the normalization $X$ and
$W=(W_1, \dots, W_n, \phi_2,\dots,\phi_{n})$ is a vector bundle on $R$.
From the remark \ref{type} it follows that $W_{_i}$ is the bundle $\mathcal{O}^{n-1}\oplus\mathcal{O}(1)$ on $R_i$ and
$\phi_1:V_{p_1}\xrightarrow{\cong} W_{_{1,p_1}}$, $\phi_{i+1}:W_{i,p_{i+1}}\xrightarrow{\cong} W_{_{i+1,p_{i+1}}}$ and
$\phi_{n+1}:V_{p_{n+1}}\xrightarrow{\cong} W_{_{n,p_{n+1}}}$.
Similarly we write the Gieseker bundle $\mathcal{E}'$ as a tuple $(V', W, \phi'_1, \phi'_{n+1})$.
Notice we have kept the bundle $(W_1, \dots, W_n, \phi_2,\dots,\phi_{n})$ fixed on the chain $R$.
This we can do because of lemma \ref{Uni}. We write $W=\oplus_{i=1}^n \mathcal{L}_{_i}$ as in lemma \ref{Uni}. Moreover, we can assume that the identifications $\phi_2,\dots,\phi_{n}$ are diagonal matrices with respect to a choice of basis
whose elements are elements of the corresponding fibres of $\mathcal{L}_{_{i,p_j}}$.

Now suppose that the corresponding parabolic triples are isomorphic
i.e., $(V, F_{\bullet}, G_{\bullet})\xrightarrow{\psi} (V', F'_{\bullet}, G'_{\bullet})$.
By composition we get isomorphisms $\phi'_{1}\circ \psi_{p_1}\circ \phi_1^{-1}: W_{p_1}\rightarrow W_{p_1}$ and
$\phi'_{n+1}\circ \psi_{p_{n+1}}\circ \phi_{n+1}^{-1}: W_{p_{n+1}}\rightarrow W_{p_{n+1}}$.
We denote them by $b^+$ and $b^-$ respectively. It is easy to see that $b^+$ and $b^-$ fixes the corresponding full flags.
Therefore $b^+\in B^+$ and $b^-\in B^-$. We write $b^+=t^+\cdot u^+$ and $b^-=t^-\cdot u^-$,
where $t^+, t^-\in T$ and $u^+\in U^+$, $u^-\in U^-$.
A priori $t^+$ and $t^-$ may not be the same. We want to change $\mathcal{E}$ with-in its equivalence class by the action of this automorphism of the chain so that the new $t^+$ and $t^-$ coincide.
In Lemma \ref{mor}, we have shown that two equivalent Gieseker bundles produce the same parabolic tuple.

Now choose any $(\lambda_1,\dots,\lambda_n)\in \times_{i=1}^n \mathbb{G}_m=:Aut(R;p_1,p_{n+1})$. Then

$(\lambda_1,\dots,\lambda_n)^* \mathcal E=(\lambda_1,\dots,\lambda_n)^* (V, W_1, \dots, W_n, \phi_1, \phi_2, \phi_3,\dots, \phi_{n}, \phi_{n+1})$

$\cong (V, W_1, \dots, W_n, \phi_1, \phi_2, \phi_3,\dots, \phi_{2k}, \dots, \phi_{n}, D_{1,2,\dots,n}(\lambda_1^2, \lambda_2^{-2},\dots, \lambda_n^{(-1)^{(n+1)}2})\phi_{n+1})$ (Lemma \ref{Action}).

By definition, it lies in the equivalence class of $\mathcal{E}$.
By composition we get isomorphisms $\phi'_{1}\circ \psi_{p_1}\circ \phi_1^{-1} : W_{p_1}\rightarrow W_{p_1}$ and
$\phi'_{n+1}\circ \psi_{p_{n+1}}\circ D_{1,2,\dots,n}(\lambda_1^2, \lambda_2^{-2},\dots, \lambda_n^{(-1)^{(n+1)}2})^{-1}\cdot \phi_{n+1}^{-1}: W_{p_{n+1}}\rightarrow W_{p_{n+1}}$. We set $b'^+:=b^+$ and $b'^-:=D_{1,2,\dots,n}(\lambda_1^2, \lambda_2^{-2},\dots, \lambda_n^{(-1)^{(n+1)}2})^{-1}\cdot b^-$.
Then $b'^+\in B^+$ and $b'^-\in B^-$. Now $b'^+=t'^+\cdot u'^+$ and $b'^-=t'^-\cdot u'^-$,
where $t'^+=t^+$, $u'^+=u^+$, $t'^-:=D_{1,2,\dots,n}(\lambda_1^2, \lambda_2^{-2},\dots, \lambda_n^{(-1)^{(n+1)}2})^{-1}\cdot t^-$
and $u'^-\in U^-$. We can choose $(\lambda_1,\dots,\lambda_n)$ in such a way that $t'^+= t'^-$.
Therefore by Lemma \ref{3.8}, we achieve an automorphism $\sigma'=(\lambda_1,\cdots,\lambda_n)$ of $W$
whose evaluation at points $p_1$
and $p_{_{n+1}}$ are $b'^+$ and $b'^-$. Let $\sigma$ be the automorphism of $X_n$ obtained by gluing the identity morphism on $X$
with the automorphism $\sigma'$ on $R$.
Then we have $\sigma^* \mathcal E\cong \mathcal E'$.
Therefore $\mathcal E$ is equivalent to $\mathcal E'$.
\end{proof}



\section{Torelli type theorem}
\begin{thm}\label{main}
Let $Y$ and $Y'$ are two irreducible projective nodal curves of genus $g\geq 2$ with single nodes $p$ and $p'$.
Let $\mathcal{M}_Y$ and $\mathcal{M}_{Y'}$ be the moduli space of stable Gieseker vector bundles of rank $2$ and
degree $d$ and $(2,d)=1$. If $\mathcal{M}_Y\cong \mathcal{M}_{Y'}$, then $Y\cong Y'$.

Suppose $Y$ and $Y'$ are curves of genus strictly greater than $3$. Let $\mathcal{M}_{_{Y,n,d}}$ and $\mathcal{M}_{_{Y',n,d}}$ be the moduli space of stable Gieseker vector bundles of rank $n(\geq 2)$ and
degree $d$ and $(n,d)=1$. If $\mathcal{M}_{_{Y,n,d}}\cong \mathcal{M}_{_{Y',n,d}}$, then $Y\cong Y'$.
\end{thm}
\begin{proof}
Using Proposition \ref{Strat}, we have stratifications (by closed sub-varieties) of $\mathcal{M}_{_{Y,n,d}}$ and $\mathcal{M}_{_{Y',n,d}}$:

\begin{equation}
\mathcal{M}_{_{Y,n,d}}^0:=\mathcal{M}_{_{Y,n,d}}\supset \mathcal{M}_{_{Y,n,d}}^1 \supset \dots \supset \mathcal{M}_{_{Y,n,d}}^n\supset \mathcal{M}_{_{Y,n,d}}^{n+1}:=\emptyset,
\end{equation}
where singular locus of $\mathcal{M}_{_{Y,n,d}}^k$ is $\mathcal{M}_{_{Y,n,d}}^{k+1}$ for $k=0,\dots, n$,
and
\begin{equation}
\mathcal{M}_{_{Y',n,d}}^0:=\mathcal{M}_{_{Y',n,d}}\supset \mathcal{M}_{_{Y',n,d}}^1 \supset \dots\supset \mathcal{M}_{_{Y',n,d}}^n\supset \mathcal{M}_{_{Y',n,d}}^{n+1}:=\emptyset,
\end{equation}
where singular locus of $\mathcal{M}_{_{Y',n,d}}^k$ is $\mathcal{M}_{_{Y',n,d}}^{k+1}$ for $k=0,\dots, n$.

Then the isomorphism $\mathcal{M}_{_{Y,n,d}}\cong \mathcal{M}_{_{Y',n,d}}$ induces isomorphisms $\mathcal{M}_{_{Y,n,d}}^k\cong \mathcal{M}_{_{Y',n,d}}^k$ for all $k$. In particular, $\mathcal{M}_{_{Y,n,d}}^n\cong \mathcal{M}_{_{Y',n,d}}^n$ i.e., the "most singular loci" of these two varieties are isomorphic.

Using Theorem \ref{main1}, we have
\[FL_n(\mathcal{U}_{p_1})\times_{\mathcal{M}_{_{VB_{_{X, n, d-n}}}}}FL_n(\mathcal{U}_{p_2})\cong FL_n(\mathcal{U}'_{{p_1}'})\times_{\mathcal{M}_{_{VB_{_{X', n, d-n}}}}}FL_n(\mathcal{U}'_{{p_2}'}),\]

where $X$ and $X'$ denote the normalizations of $Y$ and $Y'$ respectively and $\{p_1, p_2\}$ and $\{p'_1, p'_2\}$ are the pre-images of the nodes $p$ and $p'$ under the normalization maps.
Let us denote this isomorphism by $\Phi$.
Let us denote the determinant morphisms $\mathcal{M}_{_{VB_{_{X, n, d-n}}}}\rightarrow Pic_{_{X, d-n}}$ and
$\mathcal{M}_{_{VB_{_{X', n, d-n}}}}\rightarrow Pic_{_{X', d-n}}$ by $Det$ and $Det'$. Thus we have morphisms
$FL_n(\mathcal{U}_{p_1})\times_{\mathcal{M}_{_{VB_{_{X, n, d-n}}}}}FL_n(\mathcal{U}_{p_2})\to \mathcal{M}_{_{VB_{_{X, n, d-n}}}}\stackrel{Det}{\rightarrow }Pic_{_{X, d-n}}$
and $FL_n(\mathcal{U}'_{{p'_1}})\times_{\mathcal{M}_{_{VB_{_{X', n, d-n}}}}}FL_n(\mathcal{U}'_{p'_2})\to \mathcal{M}_{_{VB_{_{X', n, d-n}}}}\stackrel{Det'}{\rightarrow }Pic_{_{X', d-n}}$.
By abuse of notation, we again denote these morphisms by $Det$ and $Det'$.
Let $\xi\in Pic_{_{X, d-n}}$.
Note that $Det^{-1}(\xi)\simeq FL_n(\mathcal{U}_{p_1})\times_{\mathcal{M}_{_{VB_{_{X, n, d-n,\xi}}}}}FL_n(\mathcal{U}_{p_2})$.
Since $\mathcal{M}_{_{VB_{_{X, n, d-n,\xi}}}}$ is a simply connected variety and
$FL_n(\mathcal{U}_{p_1})\times_{\mathcal{M}_{_{VB_{_{X, n, d-n,\xi}}}}}FL_n(\mathcal{U}_{p_2})$
is a $FL_n\times FL_n$- bundle over it therefore it is also a simply connected variety, where $FL_n$ denotes the variety of full-flags of a $n$-dimensional vector space. Now consider
the morphism $Det'\circ\Phi: Det^{-1}(\xi)\rightarrow Pic_{_{X', d-n}}$.
Since $Pic_{_{X', d-n}}$ is an abelian variety and $Det^{-1}(\xi)$ is a simply connected projective variety
the map $Det'\circ\Phi$ is actually a constant map. Let us denote the image by $\xi'$.
Therefore we have
$FL_n(\mathcal{U}_{_{p_1}})\times_{_{\mathcal{M}_{_{VB_{n,d-n, \xi}}}}} FL_n(\mathcal{U}_{_{p_2}})=Det^{-1}(\xi)\cong Det'^{-1}(\xi')=FL_n(\mathcal{U}'_{p_1})\times_{\mathcal{M}_{_{VB_{_{X', n, d-n,\xi'}}}}}FL_n(\mathcal{U}'_{p_2})$.

Using \cite[Theorem 3.2]{VIS} and \cite[Theorem 4.6 (Torelli Theorem)]{DT II}, we conclude that there is an isomorphism $X\rightarrow X' $ such that the image of $\{p_1,p_2\}$ is $\{p_1',p_2'\}$. Therefore we have $Y\cong Y'$.
\end{proof}

\end{document}